\documentclass[10pt]{amsart}
\usepackage{mathrsfs}
\usepackage{amstext}
\usepackage{amscd}
\usepackage{latexsym}
\usepackage{amsfonts}
\usepackage{amssymb}
\usepackage{amsmath}
\usepackage{amsthm}
\usepackage{color}
\usepackage{multicol}
\usepackage{pgfplots}
\usepackage{multirow}
\usepackage{longtable}

\usepackage[cp1251]{inputenc}
		\theoremstyle{plain}
		\newtheorem{theorem}{Theorem}[section]
		\newtheorem{lemma}[theorem]{Lemma}
		\newtheorem{proposition}[theorem]{Proposition}
		
		\newtheorem{conjecture}[theorem]{Conjecture}		
		
		\theoremstyle{definition}
		\newtheorem{definition}[theorem]{Definition}
		
		\newtheorem{remark}[theorem]{Remark}

		\pgfplotsset{compat=1.16}

\begin{document}

\title[Integral Laplacian graphs, II]{Integral Laplacian graphs with a unique double Laplacian eigenvalue, II}			
\author[A.~Hameed]{Abdul Hameed}
			
\address{School of Mathematical Sciences, Shanghai Jiao Tong University, Shanghai, P.R.China}
\email{abdul-hameed\_211@sjtu.edu.cn}
			
\author[M.~Tyaglov]{Mikhail Tyaglov}
\address{School of Mathematical Sciences, CMA-Shanghai, and MOE-LSC, Shanghai Jiao Tong University, Shanghai, P.R.China}
\address{Department of Mathematics and Computer Sciences, St. Petersburg State University, St.Peterburg, 199178, Russia}
\email{tyaglov@mail.ru, tyaglov@sjtu.edu.cn}


\begin{abstract}
The set~$S_{\{i,j\}_{n}^{m}}=\{0,1,2,\ldots,m-1,m,m,m+1,\ldots,n-1,n\}\setminus\{i,j\},\quad 0<i<j\leqslant n$, is called Laplacian realizable if there exists a simple connected graph $G$ whose Laplacian spectrum is~$S_{\{i,j\}_{n}^{m}}$. In this case, the graph $G$ is said to realize $S_{\{i,j\}_{n}^{m}}$. In this paper, we completely describe graphs realizing the sets~$S_{\{i,j\}_{n}^{m}}$ with $m=1,2$ and determine the structure of these graphs.
\end{abstract}

\maketitle
			
\setcounter{equation}{0}
			


	\setcounter{equation}{0}
	\section{Introduction}\label{Introduction}
	
	In the present paper, we continue our previous work~\cite{AhMt_2022} where we studied graphs whose Laplacian spectrum is
	\begin{equation*}
		S_{\{i,j\}_{n}^{m}}=\{0,1,2,\ldots,m-1,m,m,m+1,\ldots,n-1,n\}\setminus\{i,j\},
		\quad 0<i<j\leqslant n,
	\end{equation*}
	with $m=n-1,n$. Here we cover the cases $m=1,2$ and give a complete description of the correspondent graphs but certain cases directly related to the so-called $S_{n,n}$-conjecture~\cite {FallatKirkland_et_al_2005}. In this work we follow the notations and preliminaries of the
	work~\cite{AhMt_2022}, however, some of them will be stated here for convenience of the reader.

	Let $G=(V(G),E(G))$ be a simple graph~(without loops or multiple edges) where $V(G)={\{v_{1},v_{2},\ldots,v_{n}}\}$ is its vertex set and  $E(G)={\{e_{1},e_{2},\ldots,e_{r}}\}$ its edge set. The entries of its Laplacian matrix are defined as follows
	\begin{equation*}
		l_{ij} =
		\begin{cases}
			&d_{i},\quad\text{if}\quad \ \ i=j,\\
			&-1,\ \ \ \text{if}\quad \ \  i\neq j\ \text{and}\  v_{i}\sim v_{j}, \\
			&\ 0,\ \ \quad\text{otherwise},
		\end{cases}
	\end{equation*}
	where $v_{i}\sim v_{j}$ means that the vertices $v_{i}$ and $v_{j}$ are adjacent.
	
	The Laplacian matrix $L(G)$ is positive semidefinite and singular, see e.g.~\cite{Zhang_2007}. A graph $G$ whose Laplacian matrix has  integer eigenvalues is called \textit{Laplacian integral}. As we noticed in~\cite{AhMt_2022}, there are many famous families of Laplacian integral graphs and we refer the reader to the works~\cite{Balinska_2002,Lima_et_al_2007,FallatKirkland_et_al_2005,GroneMerris_2008,HammerKelmans_1996,Kirkland_2005,Kirkland_2007,Kirkland_2008,Kirkland_et_al_2010,Merris.1_1994,Merris_1997,Merris_1997.1} and references therein.
	
	One of the most interesting families of Laplacian integral graphs considered by S.\,Fallat et al.~\cite{FallatKirkland_et_al_2005} is defined as follows: The set
	\begin{equation*}\label{set_S_i,n}
		S_{i,n}=\{0,1,2,\ldots,n-1,n\}\setminus\{i\},\;\;i\leqslant n,
	\end{equation*}
	is called Laplacian realizable if there exists a simple connected graph~$G$ whose Laplacian spectrum is~$S_{i,n}$. We also say that $G$ realizes $S_{i,n}$. In~\cite{FallatKirkland_et_al_2005} the authors established realizability of sets  and completely described the graphs realizing~$S_{i,n}$. In addition, it is also conjectured in~\cite{FallatKirkland_et_al_2005} that the set $S_{n,n}$ is not Laplacian realizable for any $n\geqslant 2$. . This problem is now known
	as the $S_{n,n}$-conjecture and states that $S_{n,n}$ is \textit{not} Laplacian realizable for every $n\geqslant2$. This conjecture
	was proved for $n\leqslant11$, for prime $n$, and for $n\equiv 2,3\mod4$ in~\cite{FallatKirkland_et_al_2005}. Later, Goldberger
	and Neumann~\cite{GoldbergerNeumann_2013} showed that the conjecture is true for $n\geqslant6,649,688,933$. The authors of the present
	work established~\cite{AhZuMt_2021} that if a graph is the Cartesian product of two other graphs, then it does not realize~$S_{n,n}$.

	
	As a way of investigating the class of Laplacian integral graphs, the authors of the present work extended this concept further and studied a certain class of Laplacian integral graphs introduced in~\cite{AhMt_2022} and is defined as follows.
	\begin{definition}\label{Def.S_i.j.n}
		A graph $G$ is said to realize the set
		\begin{equation}\label{set_S_i.j.n}
			S_{\{i,j\}_{n}^{m}}=\{0,1,2,\ldots,m-1,m,m,m+1,\ldots,n-1,n\}\setminus\{i,j\}
		\end{equation}
		for some $i$ and $j$, $i< j\leqslant n$,
		if its Laplacian spectrum is the set $S_{\{i,j\}_{n}^{m}}$. In this case, the set $S_{\{i,j\}_{n}^{m}}$ is called Laplacian realizable. So the set $S_{\{i,j\}_{n}^{m}}$ does not contain the numbers $i$ and $j$, while some number~$m$ (and only this number) is doubled.
		
	\end{definition}
	
	\vspace{2mm}
	
	The present work is the second part of our research on the set $S_{\{i,j\}_{n}^{m}}$. In the first part~\cite{AhMt_2022}, we considered graphs realizing the sets $S_{\{i,j\}_{n}^{m}}$ for $m=n-1$ and $m=n$, and completely described them. Moreover, we conjectured that the graphs realizing sets $S_{\{i,n\}_{n}^{m}}$~(that is, without $n$) may not exist for large~$n$. In particular, we believe that for $n\geqslant 9$ the sets $S_{\{i,n\}_{n}^{m}}$ are not Laplacian realizable. In this paper, we continue our study and consider the cases $m=1$ and $m=2$.
	
	First, we note that graphs realizing sets~$S_{\{i,j\}_{n}^{1}}$ and $S_{\{i,j\}_{n}^{2}}$ exist for small~$n$. The Laplacian spectra of all the graphs of~order up to $5$ are listed in~\cite[p.~286--289]{CvetkovicRowlinson_2010}. From that list it follows that for $n\leqslant5$ the only Laplacian realizable~$S_{\{i,j\}_{n}^{1}}$ sets are $S_{\{2,3\}_{4}^{1}}$ and~$S_{\{2,4\}_{5}^{1}}$. In Figure~\ref{Figure.doubl.1}, the graph~$G_{1}$ is the star graph~$K_{1,3}$ on~$4$ vertices realizing~$S_{\{2,3\}_{4}^1}$ (see Table~\ref{Table.1} in Appendix). The graph $G_2$ realizing the set $S_{\{2,4\}_{5}^1}$  is of the form $\left(K_2\cup 2K_1\right)\vee K_1$  (see Table~\ref{Table.1}). Similarly, for $n\leqslant 5$, the only Laplacian realizable~$S_{\{i,j\}_{n}^{2}}$ sets are~$S_{\{1,3\}_{4}^{2}}$ and~$S_{\{1,4\}_{5}^{2}}$. In Figure~\ref{Figure.doubl.2}, the graphs~$G_{3}$ and $G_4$ are the complete bipartite graphs~$K_{2,2}$~(or the cycle $C_4$) and $K_{2,3}$  on $4$ and $5$ vertices respectively, realizing~$S_{\{1,3\}_{4}^2}$ and $S_{\{1,4\}_{5}^2}$ respectively~(see Table~\ref{Table.2}).

	\begin{figure}[h]
		\begin{center}
			\begin{tikzpicture}[scale=0.9]
				
				\filldraw[fill=black] (2,1) circle(3.5pt);
				\filldraw[fill=black] (3,0.6) circle(3.5pt);
				\filldraw[fill=black] (4,1) circle(3.5pt);
				\filldraw[fill=black] (3,2) circle(3.5pt);
				\draw[line width=1.5pt] (2,1) -- (3,2);
				\draw[line width=1.5pt] (3,0.6) -- (3,2);
				\draw[line width=1.5pt] (4,1) -- (3,2);
				
				\draw node at (3,-0.5){${G_1}$};
				
				\filldraw[fill=black] (8,0.6) circle(3.5pt);
				\filldraw[fill=black] (10,0.6) circle(3.5pt);
				\filldraw[fill=black] (9,1.5) circle(3.5pt);
				\filldraw[fill=black] (10,2) circle(3.5pt);
				\filldraw[fill=black] (8,2) circle(3.5pt);
				
				\draw[line width=1.5pt] (8,0.6) -- (9,1.5);
				\draw[line width=1.5pt] (10,0.6) -- (9,1.5);
				\draw[line width=1.5pt] (8,0.6) -- (10,0.6);
				\draw[line width=1.5pt]  (8,2)-- (9,1.5);
				\draw[line width=1.5pt]  (10,2)-- (9,1.5);
				\draw node at (9,-0.5){${G_2}$};
			\end{tikzpicture}
		\end{center}
		
		\vspace{-0.10cm}
		\caption{Graphs $G_1$ and $G_2$ realizing $S_{\{2,3\}_{4}^1}$ and $S_{\{2,4\}_{5}^1}$, respectively.}\label{Figure.doubl.1}
		\vspace{4mm}
	\end{figure}
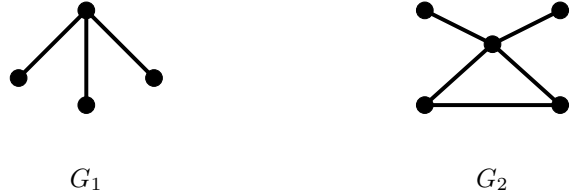
	
	\vspace{2mm}
	
	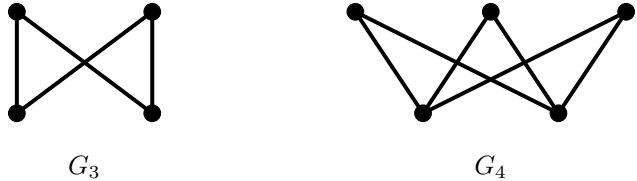
\begin{figure}[h]
		\begin{center}
			\begin{tikzpicture}[scale=0.9]
				
				\filldraw[fill=black] (2,0) circle(3.5pt);
				\filldraw[fill=black] (4,0) circle(3.5pt);
				\filldraw[fill=black] (2,1.5) circle(3.5pt);
				\filldraw[fill=black] (4,1.5) circle(3.5pt);

				\draw[line width=1.5pt] (2,1.5) -- (2,0);
				\draw[line width=1.5pt] (4,1.5) -- (2,0);
				\draw[line width=1.5pt] (4,1.5) -- (4,0);
				\draw[line width=1.5pt]  (2,1.5) -- (4,0);
				
				\draw node at (3,-0.8){${G_3}$};	
				
				\filldraw[fill=black] (8,0) circle(3.5pt);
				\filldraw[fill=black] (10,0) circle(3.5pt);
				\filldraw[fill=black] (7,1.5) circle(3.5pt);
				\filldraw[fill=black] (9,1.5) circle(3.5pt);
				\filldraw[fill=black] (11,1.5) circle(3.5pt);
				
				\draw[line width=1.5pt] (7,1.5) -- (10,0);
				\draw[line width=1.5pt] (9,1.5) -- (10,0);
				\draw[line width=1.5pt] (11,1.5) -- (10,0);
				\draw[line width=1.5pt]  (8,0) -- (7,1.5);
				\draw[line width=1.5pt] (11,1.5) -- (8,0);
				\draw[line width=1.5pt] (8,0) -- (9,1.5);
				\draw node at (9,-0.8){${G_4}$};
				
			\end{tikzpicture}
		\end{center}
		\vspace{0.3cm}
		\caption{Graphs $G_3$ and $G_4$ realizing $S_{\{1,3\}_{4}^2}$ and $S_{\{1,4\}_{5}^2}$, respectively.}\label{Figure.doubl.2}
		\vspace{3mm}
	\end{figure}

	Considering the case $m=1$, we show that the set $S_{\{i,j\}_{n}^1}$ is Laplacian realizable only if $j=n-1$, that is, only $S_{\{i,n-1\}_{n}^1}$ is Laplacian realizable for certain $i$, Theorem~\ref{double.eig.1}. Further, we list all such $i$ for fixed $n$ and  $j=n-1$, i.e., we find all the Laplacian realizable sets of kind $S_{\{i,n-1\}_{n}^1}$, Theorem~\ref{Case_double.m=1}. We also present an algorithm for constructing graphs realizing the sets $S_{\{i,n-1\}_{n}^1}$, Theorem~\ref{Constru.Double.Eig.1}. For the case $m=2$, we show that if $i>1$, then $j>n-3$, Theorem~\ref{condition.2.doubl}, and list all such $i$ for given $j$ considering  $j=n-2$ and $n-1$ separately, Theorems~\ref{Case_double.m=2} and~\ref{Thm.case.2.m=2}. Theorems~\ref{Thm.case.1.m=2} and~\ref{Theorem.S.1.j.2_constr} describe the structure of graphs realizing the sets~$S_{\{i,j\}_n^2}$ for~$j=n-2,n-1$. If $i=1$ then for all admissible~$j$, the set~$S_{\{1,j\}_n^2}$ is Laplacian realizable only if $G=K_1\vee F$, where the graph~$F$ realizes~$S_{\{j-1,n-1\}_{n-1}^1}$, Theorem~\ref{Thm.Double.2.i=1}. However, for such case we believe that it may not exists for large~$n$, $n\geqslant6$. Tables~1 and 2 in Appendix~\ref{Tables.const} illustrate the cases $m=1$ and $m=2$.

	As a result of our investigation on the set~$S_{\{i,j\}_{n}^m}$,
	we conclude that the values of $m$ are closely related to one another. For instance, if $G$ realizes~$S_{\{i,j\}_{n}^n}$~($m=n$) then one can obtain graphs realizing the sets~$S_{\{i,j\}_{n}^{n-1}}$~($m=n-1$) by using certain graph operations such as union, join and complement. Similarly, the case $m=n-2$ can be obtained from the case $m=n-1$ and so on. On the other hand, if $G$ realizes $S_{\{i,j\}_{n}^m}$ for $m=n$, then using certain graph operations, we obtain the graph realizes $S_{\{i,j\}_{n}^m}$ for $m=1$. Similarly, $m=2$ can be obtained from the case $m=n-1$. So, either of the cases can be obtained by using graph operations. However, it is not clear whether the operations on graphs cover all the sets realizing~$S_{\{i,j\}_{n}^m}$ for particular value of $m$.
	
	The paper is organized as follows. In Section~\ref{section:preliminaries}, we introduce some basic definitions and review some note worthy results from the literature that we use in this work. We also prove some auxiliary theorems. In Section~\ref{section.case.m=1}, a complete characterization of all the graphs with double Laplacian eigenvalue $m=1$ is given. The graphs with double Laplacian eigenvalue $m=2$ are discussed in Section~\ref{section.case.m=2}. We summarize this work in Section~\ref{Con.section}. Finally, in Appendix~\ref{Tables.const}, we list all the Laplacian realizable sets $S_{\{i,j\}_{n}^{1}}$  and $S_{\{i,j\}_{n}^{2}}$ for $n=4,5,6,7,8$. The associated graphs realizing those sets are presented.

	\setcounter{equation}{0}
	\section{Preliminaries}\label{section:preliminaries}
	
	An \textit{isolated} vertex is a vertex of degree zero denoted by $K_1$, while a \textit{pendant} vertex is a vertex of degree one. The \textit{complement} of a simple undirected graph $G$ denoted by $\overline{G}$ is a simple graph on the same set of vertices as $G$ in which two vertices are adjacent if and only if they are not adjacent~in~$G$. Given two disjoint graphs $G_1$ and $G_2$, the \textit{union} of these graphs, $G_1\cup G_2$, is the graph formed from the unions of the edges and vertices of the graphs $G_1$~and~$G_2$. The \textit{join} of the graphs $G_1$ and $G_2$, $G_1\vee G_2$, is the graph formed from  $G_1\cup G_2$ by adding all possible edges between vertices in $G_1$ and vertices in $G_2$, that is, $G_{1}\vee G_{2}=
	\overline{(\overline{G_{1}}\cup \overline{G_{2}})}$.
	
	We denote by $0=\mu_1\leqslant\mu_2\leqslant\ldots\leqslant\mu_n$ the \textit{Laplacian eigenvalues} of a graph~$G$. It is easy to see from the form of the Laplacian matrix that the~Laplacian spectrum of the union of two graphs is the union of their Laplacian spectra. The largest eigenvalue of the Laplacian matrix, is denoted by $\rho(G)$. The second smallest eigenvalue $\mu_2$ of $L(G)$ is usually known as
	the \textit{algebraic connectivity} of~$G$ denoted by $a_{G}$. The \textit{vertex connectivity} of a connected graph~$G$ is the minimum number of vertices whose removal disconnect~$G$.
	
	The following facts provide information on the Laplacian largest eigenvalue and the Laplacian spectrum of the complement of graph.
	\begin{theorem}[\cite{Merris_1994}]\label{Them.max.eig}
		Let $G$ be a simple graph on $n$ vertices. Then $\rho(G)\leqslant n$.
	\end{theorem}	
	\begin{theorem}[\cite{CvetkovicRowlinson_2010,Merris_1994}]\label{Thm.comp.spect}
		Let $G$ be a graph with $n$ vertices with Laplacian eigenvalues
		$$
		0=\mu_{1}\leqslant\mu_{2}\leqslant\mu_{3}\leqslant\cdots\leqslant\mu_{n-1}\leqslant\mu_{n}
		$$
		Then the Laplacian eigenvalues of the complement of~$G$ are the following
		$$
		0\leqslant n-\mu_{n}\leqslant n-\mu_{n-1}\leqslant\cdots\leqslant n-\mu_{3}\leqslant n-\mu_{2}.
		$$
	\end{theorem}	
	\vspace{4mm}
	
	The Laplacian spectra of the disjoint union and join of graphs are stated in the following theorems.
	\begin{theorem}[\cite{CvetkovicRowlinson_2010}]\label{Spect.disj.union}  If $G$ is the disjoint union of graphs $G_{1}, G_{2},\dots,G_{k}$, then it's Laplacian characteristic polynomial is
		\begin{equation*}
			\chi(G,\mu)=\prod_{k=1}^{n}\chi({G_{k}},\mu)	
		\end{equation*}
	\end{theorem}
	\begin{theorem}[Kelman's]\label{Join.Thm}
		Let $G$ and $H$ be two graphs of~order $n$ and $m$, respectively. Suppose that the Laplacian  eigenvalues of~$G$ and $H$ are of the form
		
		$$0=\mu_{1}\leqslant\mu_{2}\leqslant\mu_{3}\ldots\leqslant \mu_{n-1}\leqslant\mu_{n} \quad\;\;\mathrm{and}
		$$	
		
		$$
		0=\lambda_{1}\leqslant\lambda_{2}\leqslant\lambda_{3}\ldots\leqslant\lambda_{m-1}\leqslant\lambda_{m}\quad\;\;\mathrm{respectively}
		$$
		then the Laplacian spectrum of $G\vee H$ is of the form

		\begin{equation}\label{join.spectrum.1}
			\{0,m+\mu_{2},m+\mu_{3},\ldots,m+\mu_{n-1},m+\mu_{n},n+\lambda_{2},n+\lambda_{3},\ldots,n+\lambda_{m-1},n+\lambda_{m}, n+m\}.
		\end{equation}
		
	\end{theorem}

	We cite the above theorem in the form given in~\cite{Merris_1998}. Note that the eigenvalues in \eqref{join.spectrum.1} are not in increasing order, generally speaking.
	
	The following theorem provides a necessary and sufficient condition for a graph to have $n$ as one of its eigenvalue.
	\begin{theorem}[\cite{Molitierno_2016}]\label{Thm.Join.n} Let $G$ be a connected graph of~order $n$. Then $n$ is a Laplacian eigenvalue of~$G$ if and only if $G$ is the join of two graphs.
	\end{theorem}
	If the order~$n$ is a double Laplacian eigenvalue, then the following theorem holds.
	\begin{theorem}[\cite{AhMt_2022}]\label{theorem.Join.n.n} Let $G$ be a connected graph of~order~$n$, and let $n$ be the Laplacian eigenvalue of~$G$ of multiplicity $2$. Then $G=F\vee H$ where $F$ is a join of two graphs, while $H$ is not a join. Moreover, the eigenvalue $1$ is not in the Laplacian spectrum of~$G$.
	\end{theorem}
	
	The next proposition provides a necessary and sufficient condition for a graph to have $1$ as one of its Laplacian eigenvalues.
	
	\begin{proposition}[\cite{AhZuMt_2021}]\label{eig.1.join}
		Let a graph $G$ be a join. The number $1$ is a Laplacian eigenvalue of~$G$ if and only if $G=F\vee K_{1}$ where~$F$ is a disconnected graph of~order at least $2$.
	\end{proposition}
	
	\vspace{2mm}
	
	To complement to the previous result, we establish the following.
	\begin{theorem}\label{double.eig.1}
		Let $G$ be a connected graph of~order~$n$ that is a join. Let the number $1$ be the Laplacian eigenvalue of~$G$ of multiplicity $2$. Then~$G=\left(H_1\cup 2K_{1}\right)\vee K_{1}$ where~$H_1$ is a connected graph of~order~$n-3$.
		Moreover, the number $n-1$ is not in the Laplacian spectrum of~$G$.
	\end{theorem}
	\begin{proof}
		
		Let $G$ be a join that has a Laplacian eigenvalue $1$ of multiplicity $2$.  According to Theorem~\ref{Thm.Join.n}, $G=F\vee H$ where the graph $F$ is of~order $p$ while the graph $H$ is of~order $n-p$ for some $1\leqslant p\leqslant n-1$.
		
		Let us denote the Laplacian spectra of the graphs $F$ and $H$, respectively, as follows
		\begin{equation*}
			0=\mu_{1}\leqslant\mu_{2}\leqslant\mu_{3}\ldots\leqslant \mu_{p-1}\leqslant\mu_{p}
			\quad\text{and}\quad
			0=\lambda_{1}\leqslant\lambda_{2}\leqslant\lambda_{3}\ldots\leqslant\lambda_{n-p-1}\leqslant\lambda_{n-p}.
		\end{equation*}
		Then by Theorem~\ref{Join.Thm}, the Laplacian spectrum of~$G$ has the form
		\begin{equation}\label{spectrum.join.1}			
			\{0,(n-p)+\mu_{2},\ldots,(n-p)+\mu_{p},p+\lambda_{2},\ldots,p+\lambda_{n-p},n\}.
		\end{equation}
		We remind the reader that the eigenvalues are not in the increasing order here.
		
		Since $G$ has exactly one multiple eigenvalue $1$ by assumption, \eqref{spectrum.join.1} provides only two possible situations.
		
		\vspace{2mm}
		
		\noindent {\bf 1)}~If $n-p+\mu_{2}=p+\lambda_{2}=1$, then $n-p=p=1$, so both graphs $F$ and~$H$ are single isolated vertices. Thus, $G=K_{1}\vee K_{1}$ and the Laplacian spectrum of~$G$ is equal to $\{0,2\}$, a contradiction.
		
		\vspace{2mm}
		
		\noindent {\bf 2)}~Let $n-p+\mu_{2}=n-p+\mu_{3}=1$ or $p+\lambda_{2}=p+\lambda_{3}=1$. Without loss of generality, we can suppose that $n-p+\mu_{2}=n-p+\mu_{3}=1$, so $\mu_{2}=\mu_{3}=0$ and $n-p=1$. As well, the inequality $n-p+\mu_{4}>1$ gives us~$\mu_{4}>0$. Consequently, the graph $F$ is a disconnected graph of the form $F=H_1\cup2K_{1}$, whereas the graph $H$ is an isolated vertex $K_{1}$.
		Thus, $G=(H_1\cup2K_{1})\vee K_{1}$ where~$H_1$ is a connected graph of~order~$n-3$.
		
		Now to prove that $n-1$ is not in the Laplacian spectrum of~$G$, it is enough to show that $n-p+\mu_{p}\ne n-1$. Let us suppose that $n-p+\mu_{p}=n-1$. Then $\mu_{p}=n-2$, and so the Laplacian spectrum of the graph~$F$ of~order $n-1$ has the form $\sigma_L(F)=\{0,0,0,\mu_{4},\ldots, n-3,n-2\}$.
		According to Theorem~\ref{Thm.comp.spect}, the Laplacian spectrum of $\overline{F}$ is $\sigma_{L}(\overline{F})=\{0,1,2,\ldots,n-1,n-1\}$.
		Thus, the Laplacian spectrum of the graph $\overline{F}$ of~order $n-1$ contains the eigenvalue $1$ and the double eigenvalue $n-1$. This contradicts Theorem~\ref{theorem.Join.n.n}. Therefore, the eigenvalue $n-1$ is not in the Laplacian spectrum of~$G$.
	\end{proof}	
	Now, we remind the reader some results on the sets $S_{i,n}$ established in~\cite{FallatKirkland_et_al_2005}.		
	\begin{theorem}\label{theorem:All.S_i,n}
		Suppose $n\geqslant2$.
		\begin{itemize}
			\item[(i)] If  $n\equiv 0\mod 4$, then for each $ i=1,2,3,\dots,\dfrac{n-2}{2}$, $S_{2i,n}$ is Laplacian realizable;
			\item[(ii)] If $n\equiv 1\mod 4$, then for each $ i=1,2,3,\dots,\dfrac{n-1}{2}$, $S_{2i-1,n}$ is Laplacian realizable;
			\item[(iii)] If $n\equiv 2\mod 4$, then for each $ i=1,2,3,\dots,\dfrac{n}{2}$, $S_{2i-1,n}$ is Laplacian realizable;
			\item[(iv)] If $n\equiv 3\mod 4$, then for $ i=1,2,\dots,\dfrac{n-1}{2}$, $S_{2i,n}$ is Laplacian realizable.
		\end{itemize}
	\end{theorem}

	\begin{proposition}\label{Propos:Constr.S_1.n}
		Suppose that $n\geqslant 6$ and that $G$ is a graph on $n$ vertices. Then $G$ realizes~$S_{1,n}$ if and only if $G$ is formed in
		one of the following two ways:
		\begin{itemize}
			\item[(i)] $G=(K_{1}\cup K_{1})\vee (K_{1}\cup G_{1})$, where  $G_{1}$ is a graph on $n-3$ vertices that realizes $S_{n-4,n-3};$
			\item[(ii)] $G=K_{1}\vee H$, where  $H$ is a graph on $n-1$ vertices that realizes $S_{n-1,n-1}$.
		\end{itemize}
	\end{proposition}
	
	In the sequel, we use also the following results established in~\cite{AhMt_2022}.
	\begin{theorem}[\cite{AhMt_2022}]\label{condition.n-1}
		If $S_{\{i,j\}_{n}^{n-1}}$ is Laplacian realizable then the number $i$ is either $1$ or $2$.
	\end{theorem}
	\vspace{2mm}
	\begin{theorem}[\cite{AhMt_2022}]\label{Condition.m}
		Suppose that $n\geqslant 3$ and $G$ is a graph of~order $n$ realizing $S_{\{i,j\}_{n}^{m}}$ for $i<j$. Then
		\begin{itemize}
			\item[(i)] for $n\equiv 0\;or\; 3 \mod 4$, the numbers $(i+j)$ and $m$ are of the same parities;
			
			\vspace{2mm}
			
			\item[(ii)] for $n\equiv 1\;or\;2 \mod 4$, the numbers $(i+j)$ and $m$ are of opposite parities.
			
		\end{itemize}
	\end{theorem}
	
	\vspace{3mm}

	\begin{proposition}[\cite{AhMt_2022}]\label{double.n.missing.1.2}
		The graph $G$ realizes $S_{\{1,2\}_{n}^{n}}$ if and only if $G$ is formed in one of the following two ways:
		\begin{itemize}
			\item[(i)]  $G=P_3\vee\left(K_{1}\cup H\right)$, where $H$ realizes $S_{n-5,n-4}$ and $P_{3}$ is the path graph on $3$ vertices;
			\item[(ii)]  $G=K_2\vee H$, where $H$ realizes $S_{n-2,n-2}$.
		\end{itemize}
	\end{proposition}

	\begin{theorem}[\cite{AhMt_2022}]\label{Constru.Double.Eig}
		Let $G$ be a graph of~order $n$, $n\geqslant5$.
		\begin{itemize}
			\item [(a)]  The graph $G$ realizes $S_{\{1,2\}_{n}^{n}}$ if and only if $G$ is formed in one of the following two ways:
			\begin{itemize}
				\item[(i)]  $G=P_3\vee\left(K_{1}\cup H\right)$, where $H$ realizes $S_{n-5,n-4}$ and $P_{3}$ is the path graph on $3$ vertices;
				\item[(ii)]  $G=K_2\vee H$, where $H$ realizes $S_{n-2,n-2}$.
			\end{itemize}
			\item[(b)] If $3\leqslant j\leqslant n-2$, then $G$ realizes $S_{\{1,j\}_{n}^{n}}$ if and only if $G=K_{2}\vee\left(K_{1}\cup H\right)$, where the graph $H$ realizes~$S_{j-2,n-3}$.
			\item [(c)]  The graph $G$ realizes $S_{\{1,n-1\}_{n}^{n}}$ if and only if $G$ is formed in one of the following two ways:
			\begin{itemize}
				\item[(i)]  $G=K_2\vee(K_{2}\cup H)$, where $H$ realizes $S_{2,n-4}$;
				\item[(ii)]  $G=K_2\vee(K_{1}\cup H)$, where $H$ realizes $S_{n-3,n-3}$.
			\end{itemize}
		\end{itemize}
	\end{theorem}
	
	\vspace{2mm}
	
	\begin{theorem}[\cite{AhMt_2022}]\label{Theorem.S.1.j.n-1}
		Let $G$ be a simple connected graph of order~$n$, $n\geqslant6$.
		\begin{itemize}
			\item[(i)] For $n\equiv 0$ or $1\mod 4$, the set $S_{\{1,j\}_n^{n-1}}$ is Laplacian realizable if and only if $j=2$.
			\item[(ii)] For $n\equiv 2$ or $3\mod 4$, the set $S_{\{1,j\}_n^{n-1}}$ is Laplacian realizable if and only if $j=3$.
		\end{itemize}
	\end{theorem}

	\begin{theorem}[\cite{AhMt_2022}]\label{Theorem.S.i.n-1.constr}
		Let $G$ be a graph of order $n$, $n\geqslant6$.
		\begin{itemize}
			\item [(a)]  The graph $G$ realizes $S_{\{1,2\}_{n}^{n-1}}$ if and only if $n\equiv 0$ or $1\mod 4$, and
			$G=(K_1\cup K_2)\vee(K_1\cup H)$, where~$H$ realizes $S_{n-6,n-4}$;
			\item [(b)] The graph $G$ realizes $S_{\{1,3\}_{n}^{n-1}}$ if and only if $n\equiv 2$ or $3\mod 4$, and
			$G$ is formed in one of the following two ways:
			\begin{itemize}
				\item[(i)] $G=\left(K_{1}\cup K_{1}\right)\vee\left(K_{1}\cup H\right)$, where the graph $H$ realizes
				$S_{\{1,n-4\}_{n-3}^{n-3}}$;
				\item[(ii)] $G=K_{1}\vee F$, where the graph $F$ realizes $S_{\{2,n-1\}_{n-1}^{n-2}}$.
			\end{itemize}
		\end{itemize}
	\end{theorem}

	\setcounter{equation}{0}
	\section{Graphs realizing the sets $S_{\{i,j\}_{n}^{1}}$ }\label{section.case.m=1}
	
	In this section, we describe the graphs realizing the sets $S_{\{i,j\}_{n}^{1}}$, and present an algorithm for constructing graphs realizing $S_{\{i,j\}_{n}^{1}}$.
	As we mentioned in Introduction, in~\cite[p.~286--289]{CvetkovicRowlinson_2010} the authors listed the Laplacian spectra of all the graphs of~order up to $5$. Thus, it follows that for $n\leqslant5$ the only Laplacian realizable~$S_{\{i,j\}_{n}^{1}}$ sets are $S_{\{2,3\}_{4}^{1}}$ and~$S_{\{2,4\}_{5}^{1}}$, see Table~\ref{Table.1} in Appendix~\ref{Tables.const}. So in what follows, we consider $n\geqslant 6$.
	
	First we present the following auxiliary fact.
	
	\begin{lemma}\label{S_in>S_ij,n}
		If $S_{i,n}$ is Laplacian realizable, then so is $S_{\{i+1,n+2\}_{n+3}^{1}}$.
	\end{lemma}
	Indeed, if $G$ is a connected graph of order~$n$ realizing $S_{i,n}$, then the graph~$\left(G\cup2K_{1}\right)\vee K_{1}$ realizes~$S_{\{i+1,n+2\}_{n+3}^1}$.
	
	\vspace{5mm}
	
	In the next theorem, we describe all the Laplacian realizable sets $S_{\{i,j\}_n^1}$.
	\begin{theorem}\label{Case_double.m=1}
		Suppose $n\geqslant 6$. The only Laplacian realizable sets $S_{\{i,j\}_{n}^{1}}$, $i<j<n$, are the following ones.
		\begin{itemize}
			\item[(i)] If $n\equiv 0\mod 4$, then for each $k=1,2,\ldots,\dfrac{n-2}{2}$, $S_{\{2k,n-1\}_{n}^{1}}$ is Laplacian realizable;
			\item[(ii)] If $n\equiv 1\mod 4$, then for each  $k=1,2,\ldots,\dfrac{n-3}{2}$, $S_{\{2k,n-1\}_{n}^{1}}$ is Laplacian realizable;
			\item[(iii)] If $n\equiv 2\mod 4$, then for each $ k=1,2,\ldots,\dfrac{n-4}{2}$, $S_{\{2k+1,n-1\}_{n}^{1}}$ is Laplacian realizable;
			\item[(iv)] If $n\equiv 3\mod 4$, then for each $ k=1,2,\ldots,\dfrac{n-3}{2}$, $S_{\{2k+1,n-1\}_{n}^{1}}$ is Laplacian realizable.
		\end{itemize}
	\end{theorem}
	\begin{proof}
		According to Theorem~\ref{double.eig.1}, if $S_{\{i,j\}_{n}^1}$ is Laplacian realizable, then the eigenvalue $n-1$ is not in the Laplacian spectrum of~$G$, so $j=n-1$. Consequently, the sets  $S_{\{i,j\}_{n}^{1}}$ are not Laplacian realizable if~$j<n-1$.
		\begin{itemize}	
			\item[(i)] If $n\equiv 0\mod 4$, then $n-3\equiv1\mod4$, so for each $k=1,2,\ldots,\dfrac{n-4}{2}$, $S_{2k-1,n-3}$ is Laplacian realizable by Theorem~\ref{theorem:All.S_i,n}~(ii). According to Lemma~\ref{S_in>S_ij,n}, $S_{\{2k,n-1\}_{n}^{1}}$ is Laplacian realizable for any~$k=1,2,\ldots,\dfrac{n-4}{2}$. Moreover, since~$n$ is even and $m=1$, $S_{\{2k+1,n-1\}_{n}^{1}}$ is not Laplacian realizable by Theorem~\ref{Condition.m}~(i).
			
			Let us now deal with the set $S_{\{n-2,n-1\}_{n}^{1}}$. The set $S_{2,n-4}$ is Laplacian realizable by Theorem~\ref{theorem:All.S_i,n}~(i). Suppose that a graph $F_1$ realizes $S_{2,n-4}$. If $P_3$ is the path graph on 3 vertices, then $\sigma_{L}(\overline{P_3})=\{0,0,2\}$, and $\sigma_{L}(\overline{P_3}\cup F_1)=\{0,0,0,1,2,\ldots,n-5,n-4\}$. According to Theorem~\ref{Join.Thm}, the graph $K_1\vee(\overline{P_3}\cup F_1)$ realizes $S_{\{n-2,n-1\}_{n}^{1}}$.
			\item[(ii)] If $n\equiv 1\mod 4$, then $n-3\equiv2\mod4$. Therefore, by Theorem~\ref{theorem:All.S_i,n}~(iii), the set $S_{2k-1,n-3}$ is Laplacian realizable for $k=1,2,\ldots,\dfrac{n-3}{2}$. By Lemma~\ref{S_in>S_ij,n}, $S_{\{2k,n-1\}_{n}^{1}}$ is Laplacian realizable for any $k=~1,2,\ldots,\dfrac{n-3}{2}$. As~$n$ is odd and $m=1$, therefore, by Theorem~\ref{Condition.m}~(ii), $S_{\{2k+1,n-1\}_{n}^{1}}$ is not Laplacian realizable.
			
			The case (iii) can be proved analogously.
			
			\item[(iv)] If $n\equiv 3\mod 4$, then $n-3\equiv0\mod4$. Thus, for $k=1,2,\ldots,\dfrac{n-5}{2}$, the sets $S_{2k,n-3}$ is Laplacian realizable by  Theorem~\ref{theorem:All.S_i,n}~(i). Now Lemma~\ref{S_in>S_ij,n} implies that $S_{\{2k+1,n-1\}_{n}^{1}}$ is Laplacian realizable for any~$k=1,2,\ldots,\dfrac{n-5}{2}$, while Theorem~\ref{Condition.m}~(i) gives us that $S_{\{2k,n-1\}_{n}^{1}}$ is not Laplacian realizable, since~$n$ is even and $m=1$.
			
			Next, consider the set $S_{\{n-2,n-1\}_{n}^{1}}$. By Theorem~\ref{theorem:All.S_i,n}~(i), the set $S_{2,n-4}$ is Laplacian realizable. Let a graph, say, $F_1$ realize $S_{2,n-4}$. Then for the path graph~$P_3$, the graph $K_1\vee(\overline{P_3}\cup F_1)$ realizes~$S_{\{n-2,n-1\}_{n}^{1}}$ by Theorem~\ref{Join.Thm}. Therefore, $S_{\{n-2,n-1\}_{n}^{1}}$ is Laplacian realizable.
			
		\end{itemize}	
	\end{proof}

	In the following theorem, we discuss the structure of graphs realizing $S_{\{i,n-1\}_n^{1}}$ for various possible values of~$i$. We remind the reader that the sets $S_{\{i,j\}_n^{1}}$ for $j<n-1$ are not Laplacian realizable as we found out above.
	\begin{theorem}\label{Constru.Double.Eig.1}
		Let $G$ be a graph of order~$n$, $n\geqslant6$.
		\begin{itemize}
			\item[(a)] If~\,$2\leqslant i \leqslant n-3$, then $G$ realizes $S_{\{i,n-1\}_{n}^{1}}$ if and only if $G=K_1\vee(2K_1\cup(K_1\vee\overline{H}))$, where the graph $H$ realizes the set~$S_{n-i-2,n-4}$.

			\item [(b)]  The graph $G$ realizes $S_{\{n-2,n-1\}_{n}^{1}}$ if and only if $G$ is formed in one of the following two ways:
			\begin{itemize}
				\item[(i)]  $G=K_1\vee(\overline{P_3}\cup (K_1\vee\overline{H}))$, where the graph~$H$ of order~$n-5$ realizes $S_{n-6,n-5}$;
				\item[(ii)] $G=K_1\vee(2K_1\cup \overline{H})$, where the graph~$H$  realizes~$S_{n-3,n-3}$.
			\end{itemize}
		\end{itemize}
	\end{theorem}
	
	\begin{proof}
		\
		
		\noindent(a) Suppose that $G$ realizes $S_{\{i,n-1\}_{n}^{1}}$ for $2\leqslant i\leqslant n-3$. By Theorem~\ref{Thm.comp.spect},
		one has $\sigma_L(\overline{G})=\{0\}\cup S_{\{1,n-i\}_{n-1}^{n-1}}$. Therefore, the complement of the graph $G$ has the form $\overline{G}=K_1\cup\overline{F}$, where $\overline{F}$ is connected and
		$\sigma_L(\overline{F})=S_{\{1,n-i\}_{n-1}^{n-1}}$. According to Theorem~\ref{Constru.Double.Eig}~(b), $\overline{F}$ realizes $S_{\{1,n-i\}_{n-1}^{n-1}}$ if and only if $\overline{F}=K_{2}\vee\left(K_{1}\cup H\right)$, where the graph $H$ realizes~$S_{n-i-2,n-4}$.
		Consequently, $F=2K_1\cup(K_1\vee\overline{H})$, so $G=K_1\vee(2K_1\cup(K_1\vee\overline{H}))$.
		
		Conversely, if $G=K_1\vee(2K_1\cup(K_1\vee\overline{H}))$,
		where the graph $H$ realizes the set~$S_{n-i-2,n-4}$, then from Theorems~\ref{Spect.disj.union} and \ref{Join.Thm}, it follows that $G$ realizes~$S_{\{i,n-1\}_{n}^{1}}$.
		
		\vspace{2mm}
		
		\noindent (b) Let $G$ be a graph realizing $S_{\{n-2,n-1\}_{n}^{1}}$. Then from Theorem~\ref{Thm.comp.spect}, we obtain $\sigma_L{(\overline G)}=\{0\}\cup S_{\{1,2\}_{n-1}^{n-1}}$.
		So the complement of the graph $G$ can be represented as follows $\overline{G}=K_1\cup\overline{F}$, where $\overline{F}$ is connected and
		$\sigma_L(\overline{F})=S_{\{1,2\}_{n-1}^{n-1}}$. According to Theorem~\ref{Constru.Double.Eig}~(a), the graph $\overline{F}$ must be in one of the following form:
		\begin{itemize}
			\item[(i)] $\overline{F}=P_3\vee\left(K_{1}\cup H\right)$, where $H$ realizes $S_{n-6,n-5}$ and $P_{3}$ is the path graph on $3$ vertices;
			\item[(ii)] $\overline{F}=K_2\vee H$, where $H$ realizes $S_{n-3,n-3}$.
		\end{itemize}
		
		Thus, from~(i) one gets ${F}=\overline{P_3}\cup\left(K_{1}\vee \overline{H}\right)$. According to Theorem~\ref{Thm.comp.spect}, one has $\sigma_{L}(\overline{H})=\{0\}\cup S_{1,n-6}$. So $\overline{H}$ can be represented as follows $\overline{H}=K_1\cup H_1$, where $H_1$ realizes $S_{1,n-6}$~(for the construction of $H_1$, see Proposition~\ref{Propos:Constr.S_1.n}). Thus, $G=K_1\vee(\overline{P_3}\cup (K_1\vee\overline{H}))$, where the graph $H$ of order~$n-5$ realizes $S_{n-6,n-5}$.
		
		Also from~(ii), we have $\overline{F}={K_2\vee H}$. Therefore, $F=2K_1\cup\overline{H}$, and the graph $H$ is a connected graph realizing $S_{n-3,n-3}$.
		Thus, $G=K_1\vee(2K_1\cup \overline{H})$, where $H$ realizes $S_{n-3,n-3}$.
		
		Conversely, if $G=K_1\vee(\overline{P_3}\cup (K_1\vee\overline{H}))$, where the graph $H$ of order~$n-5$ realizes $S_{n-6,n-5}$. Then from Theorems~\ref{Spect.disj.union} and \ref{Join.Thm} it follows that $G$ realizes $S_{\{n-2,n-1\}_{n}^{1}}$. Similarly, if $G=K_1\vee(2K_1\cup \overline{H})$, where~$H$  realizes $S_{n-3,n-3}$, then again by Theorems~\ref{Spect.disj.union} and \ref{Join.Thm}, the graph $G$ realizes $S_{\{n-2,n-1\}_{n}^{1}}$.
	\end{proof}	
	
	\begin{remark}
		In the above theorem, the structure of the graph~$G$ shows that there exist at least two pendant vertices in~$G$. As well, Theorems~\ref{Case_double.m=1} and~\ref{Constru.Double.Eig.1} completely resolve the existence of graphs realizing
		the spectrum $S_{\{i,j\}_n^1}$, where $2\leqslant i<j<n$. Furthermore, it is easily deduced from
		Theorem~\ref{Constru.Double.Eig.1} that if the sets~$S_{n,n}$ were not realizable for any $n$, then there is a unique graph, which realizes~$S_{\{i,j\}_n^1}$, for $2\leqslant i<j<n$.
	\end{remark}
	
	
	For the case $j=n$, we conjectured that the set~$S_{\{i,n\}_{n}^{m}}$, $n\geqslant 9$, does not exist~\cite{AhMt_2022}. However, unless otherwise is proved, we must include this case in our consideration.
	
	\begin{proposition}\label{double.1.Prop}
		Let $G$ realize~$S_{\{i,n\}_n^1}$. Then the following holds:
		\begin{itemize}
			\item[(a)] $n\geqslant 9$;
			\item[(b)] $n$ is not a prime number;
			\item[(c)] $2\leqslant\min\limits_{1\leqslant i\leqslant n}~d_i\leqslant \max\limits_{1\leqslant i\leqslant n}d_i\leqslant n-3$,\quad\,where $d_i$ is the degree of vertex~$i$.
		\end{itemize}
	\end{proposition}
	
	\begin{proof}
		Let $G$ realize~$S_{\{i,n\}_n^1}$. Then Properties~(a) and~(b) follow from~\cite[Proposition~5.2]{AhMt_2022}.
		
		If we suppose, on the contrary, that $G$ has a pendant vertex,  then the vertex connectivity of $G$ equals~$1$. At the same time, its algebraic connectivity is also $1$. So, according to~\cite[Theorem~2.1]{KirMol_2007} if the algebraic and vertex connectivity are equal, then $G$ must be a join, a contradiction.
	\end{proof}
	
	Property~(c) in the above proposition coincides with the one of the $S_{n,n}$-conjecture proposed~in~\cite[Observation~3.5]{FallatKirkland_et_al_2005}.

	\vspace{2mm}
	
	The \textit{Cartesian product} of the graphs $G_1$ and $G_2$ is the graph $G_1\times G_2$ whose vertex set is the Cartesian product $V(G_1)\times V(G_2)$, and for $v_{1},v_{2}\in V(G_1)$ and  $u_{1},u_{2}\in V(G_2)$, the vertices $(v_{1},u_{1})$ and $(v_{2},u_{2})$ are adjacent in $G_1\times G_2$ if and only if either
	\begin{itemize}
		\item[$\bullet$] $v_{1}=v_{2}$ and $\{u_{1},u_{2}\}\in E(G_1)$;
		\item[$\bullet$]${\{v_{1},v_{2}\}}\in E(G_2)$ and $u_{1}$=$u_{2}$.
	\end{itemize}	
	
	The authors of the present work established \cite{AhMt_2022} that if $G=G_1\times G_2$ (i.e., $G$ is a Cartesian product of two graphs) then it does not realize the set $S_{\{i,n\}_n^m}$ for $n\geqslant 9$ and any $m\neq i,n$. According to Proposition~\ref{double.1.Prop}, the set $S_{\{i,n\}_n^1}$ is not Laplacian realizable if $n$ is a prime number, and so the Cartesian product does not realizes the set~$S_{\{i,n\}_n^1}$ for a prime number $n$. All these facts motivated us to state the following Conjecture.
	
	\begin{conjecture}\label{conj.doub.1}
		For $n\geqslant 3$, the set~$S_{\{i,n\}_n^1}$ is not Laplacian realizable.
	\end{conjecture}

	\setcounter{equation}{0}
	\section{Graphs realizing the sets $S_{\{i,j\}_{n}^{2}}$}\label{section.case.m=2}

	This section establishes the existence of graphs realizing the sets~$S_{\{i,j\}_{n}^{2}}$. As we already mentioned in Section~\ref{Introduction},
	for $n\leqslant 5$, the only Laplacian realizable sets of type $S_{\{i,j\}_{n}^{2}}$ are $S_{\{1,3\}_4^2}$ and $S_{\{1,4\}_5^2}$, see Table~\ref{Table.2} in Appendix~\ref{Tables.const}. So throughout this section we consider $n\geqslant6$. In Sections~\ref{subsec.1}--\ref{subsec.2} we study the case $1<i<j<n$, while Section~\ref{subsec.3} is devoted to the sets $S_{\{1,j\}_{n}^{2}}$~($i=1$) and~$S_{\{i,n\}_{n}^{2}}$~($j=n$). We believe that these sets are not Laplacian realizable for large $n$.
	To proceed with the main results of this section, we first establish a relation between the numbers $i$ and $j$ to realize~$S_{\{i,j\}_{n}^{2}}$ when $i>1$ as follows:
	\begin{theorem}\label{condition.2.doubl}
		Let $G$ realize~$S_{\{i,j\}_{n}^{2}}$. If $i>1$, then $j>n-3$.
	\end{theorem}
	\begin{proof} \noindent
		Suppose on the contrary that $S_{\{i,j\}_{n}^{2}}$ is Laplacian realizable for some $j\leqslant n-3$, and let $G$ be a graph realizing this set. Clearly, the numbers $n-2$, $n-1$ and $n$ belong to $\sigma_L(G)$.  According to Theorem~\ref{Thm.comp.spect},
		\begin{equation}\label{cond.2.i>1}
			\sigma_L(\overline{G})=\{0,0,1,2,\ldots,n-j-1,n-j+1,n-i-1,n-i+1,\ldots,n-2,n-2,n-1\}.
		\end{equation}
		By Theorem~\ref{Them.max.eig} one of the components of~$\overline{G}$, say, $\overline{G_2}$ has $n-1$ vertices, so $G_1=K_1$. Now from \eqref{cond.2.i>1} one has
		$$\sigma_L(\overline{G_2})=\{0,1,2,\ldots,n-j-1,n-j+1,n-i-1,n-i+1,\ldots,n-2,n-2,n-1\}.
		$$
		This contradicts Theorem~\ref{condition.n-1}. Therefore, for $i>1$, the set~$S_{\{i,j\}_{n}^{2}}$ is Laplacian realizable only if~$j>n-3$.
	\end{proof}
	
	\vspace{2mm}
	
	The above theorem claims that for $i>1$ either $j=n-2$, $n-1$ or $n$. In the sequel, we consider these three cases separately.
	%
	
	\subsection{Graphs realizing the sets~ $S_{\{i,n-2\}_{n}^{2}}$}\label{subsec.1}

	We start with the following fact established in~\cite{AhMt_2022}.
	
	\begin{proposition}[\cite{AhMt_2022}]\label{join.union}
		Let $n\geqslant 3$, if $S_{\{i,j\}_{n}^{m}}$ is Laplacian realizable, then so is $S_{\{i+1,j+1\}_{n+2}^{m+1}}$.
	\end{proposition}

	The following theorem lists all the Laplacian realizable sets~$S_{\{i,n-2\}_n^2}$. We, remind that $i>1$.
	
	\begin{theorem}\label{Case_double.m=2}
		Suppose $n\geqslant 6$. The only Laplacian realizable sets $S_{\{i,n-2\}_{n}^{2}}$ are the following ones.
		\begin{itemize}
			\item[(i)] If $n\equiv 0\mod 4$, then for each $k=1,2,\ldots,\dfrac{n-6}{2}$, $S_{\{2k+2,n-2\}_{n}^{2}}$ is Laplacian realizable;
			\item[(ii)] If $n\equiv 1\mod 4$, then for each  $k=1,2,\ldots,\dfrac{n-5}{2}$, $S_{\{2k+2,n-2\}_{n}^{2}}$ is Laplacian realizable;
			\item[(iii)] If $n\equiv 2\mod 4$, then for each $ k=1,2,\ldots,\dfrac{n-4}{2}$, $S_{\{2k+1,n-2\}_{n}^{2}}$ is Laplacian realizable;
			\item[(iv)] If $n\equiv 3\mod 4$, then for each $ k=1,2,\ldots,\dfrac{n-5}{2}$, $S_{\{2k+1,n-2\}_{n}^{2}}$ is Laplacian realizable.
		\end{itemize}
	\end{theorem}
	
	\begin{proof}
		\
		\begin{itemize}
			\item[(i)] If $n\equiv 0 \mod 4$, then $n-2\equiv 2\mod 4$, then for $k=1,2,\ldots,\dfrac{n-6}{2}$, so the sets $S_{\{2k+1,n-3\}_{n-2}^{1}}$
			are Laplacian realizable by Theorem~\ref{Case_double.m=1}~(iii). From Proposition~\ref{join.union} it follows that $S_{\{2k+2,n-2\}_n^{2}}$
			is Laplacian realizable for each $k=1,2,\ldots,\dfrac{n-6}{2}$. At the same time, the sets $S_{\{2k+1,n-2\}_n^{2}}$ are not Laplacian
			realizable for any $k$ by Theorem~\ref{Condition.m}~(i), since the double eigenvalue $m=2$ is an even number in this case.
			\item[(ii)] For $n\equiv 1 \mod 4$, we have $n-2\equiv 3\mod 4$. Thus,  for $k=1,2,\ldots,\dfrac{n-5}{2}$, the sets $S_{\{2k+1,n-3\}_{n-2}^{1}}$
			are Laplacian realizable by Theorem~\ref{Case_double.m=1}~(iv). Now Proposition~\ref{join.union} implies that $S_{\{2k+2,n-2\}_n^{2}}$ are Laplacian realizable for $k=1,2,\ldots,\dfrac{n-5}{2}$. As $m=2$ is even, from Theorem~\ref{Condition.m}~(ii) it follows that
			the sets $S_{\{2k+1,n-2\}_n^{2}}$  are not Laplacian realizable for any $k$.
		\end{itemize}
		
		The cases (iii) and (iv) can be proved analogously with the use of Theorem~\ref{Case_double.m=1} and Proposition~\ref{join.union}.
		
	\end{proof}

	The next theorem deals with the construction of graphs realizing the sets $S_{\{i,n-2\}_{n}^{2}}$, $i>1$.
	\begin{theorem}\label{Thm.case.1.m=2}
		Let $n\geqslant 6$, and let $G$ be a connected graph of order~$n$. Then $G$ realizes
		$S_{\{i,n-2\}_n^{2}}$ if and only if $G=K_1\vee \left(K_{1}\cup H\right)$,
		where $H$ is a graph on $n-2$ vertices realizing $S_{\{i-1,n-3\}_{n-2}^{1}}$.
	\end{theorem}
	\begin{proof}
		Let $G$ realize $S_{\{i,n-2\}_{n}^{2}}$. Then $G=G_1\vee G_2$ by Theorem~\ref{Thm.Join.n},
		so $\overline{G}=\overline{G_{1}}\cup\overline{G_{2}}$. From Theorem~\ref{Thm.comp.spect} it follows that
		$\sigma_L(\overline G)=\{0\}\cup S_{\{2,n-i\}_{n-1}^{n-2}}$. Thus, by Theorem~\ref{Them.max.eig} we obtain that $G_1=K_1$,
		and $G_2$ is of order~$n-1$, so that $\sigma_L(\overline{G_2})=S_{\{2,n-i\}_{n-1}^{n-2}}$. Using Theorem~\ref{Thm.comp.spect},
		we get $\sigma_L(G_2)=\{0\}\cup S_{\{i-1,n-3\}_{n-2}^{1}}$. Again Theorem~\ref{Them.max.eig}
		gives us that $G_2=K_1\cup H$, where $H$ is a graph on $n-2$ vertices realizing~$S_{\{i-1,n-3\}_{n-2}^{1}}$.
		Consequently, $G=K_1\vee \left(K_{1}\cup H\right)$, as required.
		
		Conversely, if $G=K_{1}\vee(K_{1}\cup H)$, where $H$ is a graph on $n-2$ vertices realizing~$S_{\{i-1,n-3\}_{n-2}^{1}}$,
		then from Theorem~\ref{Spect.disj.union} and \ref{Join.Thm} it follows that $G$ realizes $S_{\{i,n-2\}_{n}^{2}}$.
	\end{proof}

	\begin{remark}
		Theorems~\ref{Case_double.m=2} and~\ref{Thm.case.1.m=2} completely resolve the existence of graphs realizing
		the spectrum~$S_{\{i,n-2\}_n^2}$, where $2\leqslant i\leqslant n-3$.
	\end{remark}

	\subsection{Graphs realizing the sets~ $S_{\{i,n-1\}_{n}^{2}}$}\label{subsec.2}
	
	To determine all the graphs realizing the set~$S_{\{i,n-1\}_{n}^{2}}$ for various possible values of $i$, $i>1$, first we establish the following auxiliary lemma.
	
	\begin{lemma}\label{compl.join.2}
		The set $S_{\{i,j\}_{n}^{m}}$ is Laplacian realizable if and only if the set $S_{\{n-j+1,n-i+1\}_{n+1}^{n+1-m}}$ is Laplacian realizable.
	\end{lemma}
	\begin{proof}
		Let $S_{\{i,j\}_{n}^{m}}$ be Laplacian realizable and let a graph~$G$ realize $S_{\{i,j\}_{n}^{m}}$. Consider the graph $H=\overline{G}\vee K_1$.
		Since the Laplacian spectrum of the graph $\overline{H}=G\cup K_1$ has the
		form $\{0\}\cup S_{\{i,j\}_{n}^{m}}$, the spectrum of $H$ is $S_{\{n-j+1,n-i+1\}_{n+1}^{n+1-m}}$ by Theorem~\ref{Thm.comp.spect}.
		
		Conversely, suppose that $S_{\{n-j+1,n-i+1\}_{n+1}^{n-m+1}}$ is Laplacian realizable, and a graph $H$ realizes it. According to Theorem~\ref{Thm.comp.spect}, we have
		\begin{equation}\label{compl.join.2.proof}
			\sigma_{L}(\overline{H})=\{0\}\cup S_{\{i,j\}_{n}^{m}}.
		\end{equation}
		Thus, $\overline{H}$ is the union of two disjoint graphs, say, $\overline H=G\cup F$, and the Laplacian spectrum of $\overline H$
		is the union of the spectra of~$G$ and $F$. Consequently, one of these graphs, say, $G$, has $n$ as its Laplacian eigenvalue, so that its order is at least $n$ by Theorem~\ref{Them.max.eig}. Since the order of $\overline H$ is $n+1$, we have $F=K_1$. Therefore, the graph~$G$ of order $n$ realizes $S_{\{i,j\}_{n}^{m}}$, according to~\eqref{compl.join.2.proof}.
	\end{proof}
	
	\vspace{2mm}
	
	Now, we are in a position to describe all the Laplacian realizable sets $S_{\{i,n-1\}_n^2}$. Remind that $i>1$.
	\begin{theorem}\label{Thm.case.2.m=2}
		Suppose $n\geqslant 6$. The only Laplacian realizable sets $S_{\{i,n-1\}_{n}^{2}}$ are the following ones.
		\begin{itemize}
			\item[(i)] If $n\equiv 0\;or\;3\mod 4$, then  $S_{\{i,n-1\}_{n}^{2}}$ is Laplacian realizable if and only if $i=n-3$.
			\item[(ii)]  If $n\equiv 1\;or\;2\mod 4$, then  $S_{\{i,n-1\}_{n}^{2}}$ is Laplacian realizable if and only if $i=n-2$.
		\end{itemize}
	\end{theorem}
	
	\begin{proof}
		\
		Let $G$ realizes $S_{\{i,n-1\}_{n}^{2}}$.
		\begin{itemize}
			\item[(i)] If $n\equiv 0\;or\;3\mod 4$, then $n-1\equiv 2\; or\;3\mod 4$, so by Theorem~\ref{Theorem.S.1.j.n-1}~(ii), the set $S_{\{1,j\}_{n-1}^{n-2}}$ is Laplacian realizable if and only if $j=3$. Using Lemma~\ref{compl.join.2}, the set  $S_{\{n-j,n-1\}_{n}^{2}}$ is also Laplacian realizable if and only if $j=3$. Thus, $S_{\{i,n-1\}_{n}^{2}}$ is Laplacian realizable only, if $i=n-3$.\\
			
			\item[(ii)]
			If $n\equiv 1\;or\;2\mod 4$, then $n-1\equiv 0\; or\;1\mod 4$. Therefore, by Theorem~\ref{Theorem.S.1.j.n-1}~(i), the set $S_{\{1,j\}_{n-1}^{n-2}}$ is Laplacian realizable if and only if $j=2$. Using Lemma~\ref{compl.join.2}, the set  $S_{\{n-j,n-1\}_{n}^{2}}$ is Laplacian realizable if and only if $j=2$. Consequently, $S_{\{i,n-1\}_{n}^{2}}$ is Laplacian realizable only, if $i=n-2$.
		\end{itemize}
	\end{proof}

	Our next result, discuss the structure of graphs realizing the sets~$S_{\{i,n-1\}_{n}^{2}}$, $i>1$ as follows.
	\begin{theorem}\label{Theorem.S.1.j.2_constr}
		Let $G$ be a graph of order~$n$, $n\geqslant6$.
		\begin{itemize}
			\item [(a)]  The graph $G$ realizes $S_{\{n-3,n-1\}_{n}^2}$ if and only if $n\equiv 0$ or $3\mod 4$,
			and
			$G$ is formed in one of the following two ways:
			
			\begin{itemize}
				\item[(i)] $G=K_1\vee({K_2}\cup (K_1\vee(2K_1\cup H_1)))$, where the graph $H_1$ of order~$n-6$ realizes $S_{1,n-6}$;
				\item[(ii)] $G=K_1\vee\left(K_{1}\cup \overline{F}\right)$, where the graph~$F$ realizes $S_{\{2,n-2\}_{n-2}^{n-3}}$.
			\end{itemize}
			\item [(b)] The graph $G$ realizes $S_{\{n-2,n-1\}_{n}^2}$ if and only if $n\equiv 1$ or $2\mod 4$,
			and
			$G=K_1\vee\left(P_3\cup(K_1\vee \overline{H})\right)$, where~$H$ realizes $S_{n-7,n-5}$.
		\end{itemize}
	\end{theorem}

	\vspace{2mm}
	
	\begin{proof}
		\noindent (a) Let $G$ realizes $S_{\{n-3,n-1\}_{n}^2}$, and let $n\equiv 0$ or $3\mod 4$. By Theorem~\ref{Thm.Join.n}, $G=G_1\vee G_2$, so $\overline{G}=\overline{G_{1}}\cup\overline{G_{2}}$. From Theorem~\ref{Thm.comp.spect} it follows that
		$\sigma_L(\overline G)=\{0\}\cup S_{\{1,3\}_{n-1}^{n-2}}$. Thus, by Theorem~\ref{Them.max.eig} we obtain that $G_1=K_1$,
		and $G_2$ is of order~$n-1$, so that $\sigma_L(\overline{G_2})=S_{\{1,3\}_{n-1}^{n-2}}$. According to Theorem~\ref{Theorem.S.i.n-1.constr}, the graph $\overline{G_2}$ can be formed in one of the following two ways.
		\begin{itemize}
			\item[(i)] $\overline{G_2}=\left(K_{1}\cup K_{1}\right)\vee\left(K_{1}\cup H\right)$, where the graph $H$ realizes
			$S_{\{1,n-4\}_{n-3}^{n-3}}$;
			\item[(ii)] $\overline{G_{2}}=K_{1}\vee F$, where the graph $F$ realizes $S_{\{2,n-2\}_{n-2}^{n-3}}$.
		\end{itemize}
		
		Thus, from~(i) one gets ${G_2}=K_2\cup\left(K_{1}\vee \overline{H}\right)$. According to Theorem~\ref{Thm.comp.spect}, one has $\sigma_{L}(\overline{H})=\{0,0\}\cup S_{1,n-5}$. Thus, $\overline{H}$ can be represented as follows $\overline{H}=2K_1\cup H_1$, where $H_1$ realizes $S_{1,n-5}$~(for the construction of~$H_1$ see Proposition~\ref{Propos:Constr.S_1.n}). Thus, $G=K_1\vee({K_2}\cup (K_1\vee(2K_1\cup H_1)))$, where the graph $H_1$ of order~$n-6$ realizes $S_{1,n-6}$.
		
		Also from~(ii), we have ${G_2}={K_1\cup \overline{F}}$.  According to Theorem~\ref{Thm.comp.spect}, one has $\sigma_{L}(\overline{F})=S_{\{n-4,n-2\}_{n-2}^1}$. Thus, $G=K_1\vee\left(K_1\cup\overline{F}\right)$  where $F$ realizes $S_{\{2,n-2\}_{n-2}^{n-3}}$~(for the construction of $H_1$ see Proposition~\ref{Propos:Constr.S_1.n}).
		
		Conversely, if $G=K_1\vee({K_2}\cup (K_1\vee(2K_1\cup H_1)))$, where the graph $H_1$ of order~$n-6$ realizes $S_{1,n-6}$. Then, from Theorem~\ref{Spect.disj.union} and \ref{Join.Thm}, it follows that $G$ realizes $S_{\{n-3,n-1\}_{n}^{2}}$. Similarly, if $G=K_1\vee\left(K_1\cup\overline{F}\right)$  where the graph $F$ realizes $S_{\{2,n-2\}_{n-2}^{n-3}}$, then again from Theorem~\ref{Spect.disj.union} and \ref{Join.Thm}, the graph $G$ realizes $S_{\{n-3,n-1\}_{n}^{2}}$.
		
		\vspace{2mm}
		
		\noindent (b)  Let $G$ realizes $S_{\{n-2,n-1\}_{n}^2}$ and let $n\equiv 1$ or $2\mod 4$. Then $G=G_1\vee G_2$ by Theorem~\ref{Thm.Join.n},
		so $\overline{G}=\overline{G_{1}}\cup\overline{G_{2}}$. From Theorem~\ref{Thm.comp.spect} it follows that
		$\sigma_L(\overline G)=\{0\}\cup S_{\{1,2\}_{n-1}^{n-2}}$. Thus, by Theorem~\ref{Them.max.eig} we obtain that $G_1=K_1$, and $G_2$ is of order~$n-1$, so that $\sigma_L(\overline{G_2})=S_{\{1,2\}_{n-1}^{n-2}}$. According to Theorem~\ref{Theorem.S.i.n-1.constr}, the graph $\overline{G_2}$ is of the following form.
		$\overline{G_2}=(K_1\cup K_2)\vee(K_1\cup H)$, where~$H$ realizes $S_{n-7,n-5}$. Thus, $G_2=P_3\cup(K_1\vee \overline{H})$. According to Theorem~\ref{Thm.comp.spect}, one has $\sigma_{L}(\overline{H})=\{0\}\cup S_{2,n-6}$. Thus, $\overline{H}$ can be represented as follows $\overline{H}=K_1\cup H_1$, where $H_1$ realizes $S_{2,n-6}$~(for the construction of $H_1$ see Proposition~\ref{Propos:Constr.S_1.n}). Consequently, $G=K_1\vee\left(P_3\cup(K_1\vee \overline{H})\right)$, where~$H$ realizes $S_{n-7,n-5}$.
		
		Conversely, if $G=K_1\vee\left(P_3\cup(K_1\vee \overline{H})\right)$, where~$H$ realizes $S_{n-7,n-5}$. Then, from Theorem~\ref{Spect.disj.union} and \ref{Join.Thm} it follows that $G$ realizes $S_{\{n-2,n-1\}_{n}^{2}}$.
	\end{proof}
	
	\subsection{Graphs realizing the sets~ $S_{\{1,j\}_{n}^{2}}$ and $S_{\{i,n\}_{n}^{2}}$}\label{subsec.3}
	
	Now we are in a position to study graphs realizing the sets~$S_{\{1,j\}_{n}^{2}}$~($i=1$) and $S_{\{i,n\}_{n}^{2}}$~($j=n$).
	Since all the correspondent graphs on less than 6 vertices have been listed, see Appendix~\ref{Tables.const}, in what follows we consider $n\geqslant 6$.
	The structure of graph realizing the set~$S_{\{1,j\}_{n}^{2}}$ is described by the following theorem.
	\begin{theorem}\label{Thm.Double.2.i=1}
		Let $G$ be a graph of order $n$. Then for each admissible $j$, $1<j<n$ the graph $G$ realizes~$S_{\{1,j\}_{n}^{2}}$ if and only if $G=K_1\vee F$ where $F$ realizes~$S_{\{j-1,n-1\}_{n-1}^{1}}$.
	\end{theorem}
	\begin{proof}
		Let $S_{\{1,j\}_{n}^{2}}$ for $1<j<n$ be Laplacian realizable by a graph $G$. Then  Theorem~\ref{Thm.comp.spect} implies
		\begin{equation}\label{cond.i=1}
			\sigma_{L}(\overline{G})=\{0,0,1,2,\ldots,n-j-1,n-j+1,\ldots,n-2,n-2\}.
		\end{equation}
		Here $\overline{G}$ is a disconnected graph of the form $\overline{G}=\overline{G_1}\cup\overline{G_2}$. By Theorem~\ref{Them.max.eig}, one of the components, say, $\overline{G_2}$, must be of order at least~$n-2$.
		
		If $\overline{G_2}$ has $n-2$ vertices, then $\overline{G_1}$ has two vertices, and therefore $\sigma_L(\overline{G_1})=\{0,2\}$. By Theorem~\ref{Thm.comp.spect}, we obtain $\sigma_L({G_1})=\{0,0\}$, so $G_1=2K_1$. Now from \eqref{cond.i=1}, we have
		$$
		\sigma_L(\overline{G_2})=\{0,1,\ldots,n-j-1,n-j+1\ldots,n-2,n-2\}.
		$$
		This contradicts Theorem~\ref{theorem.Join.n.n}.
		
		If $\overline{G_2}$ has $n-1$ vertices, then $\overline{G_1}=K_1$. Again from \eqref{cond.i=1}
		
		$$
		\sigma_L(\overline{G_2})=\{0,1,2,\ldots,n-j-1,n-j+1\ldots,n-2,n-2\},
		$$
		by Theorem~\ref{Thm.comp.spect},
		$$
		\sigma_L({G_2})=\{0,1,1,2,\ldots,j-2,j\ldots,n-3,n-2\}.
		$$
		Here the graph $G_2$ is of order~$n-1$ realizes~$S_{\{j-1,n-1\}_{n-1}^{1}}$. Thus, $G=K_1\vee F$, where $F=G_2$~realizes~$S_{\{j-1,n-1\}_{n-1}^{1}}$.
		
		Conversely, if $G=K_1\vee F$, where the graph $F$ of order~$n-1$ realizes $S_{\{j-1,n-1\}_{n-1}^{1}}$. Then from Theorems~\ref{Join.Thm} it follows that $G$ realizes $S_{\{1,j\}_{n}^{2}}$.
	\end{proof}
	Theorem~\ref{Thm.Double.2.i=1}, guarantees that if the set $S_{\{i,n\}_{n}^{m}}$ is not Laplacian realizable, then no graphs realizing~$S_{\{1,j\}_{n}^{2}}$ exist. From Theorem~\ref{Thm.Double.2.i=1} it follows that for all admissible $j$, the sets~$S_{\{1,j\}_{n}^{2}}$ are Laplacian realizable only if the Conjecture~\ref{conj.doub.1} is not true. Moreover, for $n=p+1$ where $p$ is a prime number, the set $S_{\{1,j\}_{n}^{2}}$ is not Laplacian realizable according to Proposition~\ref{double.1.Prop}. For instance, if $G$ is of order~$n=6,8,12,14,18$ etc, then it does not realize~$S_{\{1,j\}_{n}^{2}}$.
	
	\vspace{2mm}
	
	Finally, we consider the case when $j=n$. Recall that for $n\leqslant 5$, there are no graphs realizing the set~$S_{\{i,n\}_{n}^{2}}$, see Appendix~\ref{Tables.const}. As well, we conjectured in~\cite{AhMt_2022} that graphs realizing the set~$S_{\{i,n\}_{n}^{m}},n\geqslant 9$ do not exist.
	In that paper, it is also shown that the set set~$S_{\{i,n\}_{n}^{m}}$ is not Laplacian realizable for a prime number~$n$, and so the set~$S_{\{i,n\}_{n}^{2}}$ is not Laplacian realizable if~$n$ is prime. Moreover, as we mentioned in Section~\ref{section.case.m=1} above, the set~$S_{\{i,n\}_{n}^{2}}$ is not Laplacian realizable by the Cartesian product for $n\geqslant 9$. Additionally,  for $i>1$, one can see from Proposition~\ref{double.1.Prop} that $G$ has the minimum and maximum degree of the following form:
	$$ 2\leqslant\min\limits_{1\leqslant i\leqslant n}~d_i\leqslant \max\limits_{1\leqslant i\leqslant n}d_i\leqslant n-3,$$
	where $d_i$ is the degree of vertex~$i$. So the minimum degree of graphs realizing the sets~$S_{\{i,n\}_{n}^{2}}$, $i>1$ and, more generally, the sets~$S_{\{i,n\}_{n}^{m}}$ (if any) is greater than or equal to $2$ and by graph complement one can obtain the maximum degree $n-3$. If the minimum degree equals~$1$, then $G$ is a join, a contradiction. Note that this property is analogous to the one of $S_{n,n}$-conjecture proposed in~\cite[Section~3]{FallatKirkland_et_al_2005} the authors showed that for $n=8,9$ the set $S_{n,n}$ is not Laplacian realizable. We believe that the same is true for sets $S_{\{i,n\}_{n}^{m}}$. Thus, if $S_{n,n}$-conjecture is true, then our $S_{\{i,n\}_{n}^{m}}$-conjecture will also hold true. According to these observations, we believe that the set~$S_{\{i,n\}_{n}^{2}}$, is not Laplacian realizable. For more details on $S_{\{i,n\}_{n}^{m}}$-conjecture, we refer the reader to~\cite[Section~5]{AhMt_2022}.

	\setcounter{equation}{0}
	\section{Conclusion}\label{Con.section}
	
	In~\cite{AhMt_2022}, we established the existence of graphs realizing the sets~$S_{\{i,j\}_n^m}$ for $m=n$ and $m=n-1$~and completely described them.
	The present work continues our study on the realizability of sets~$S_{\{i,j\}_n^m}$ for cases~$m=1,2$. We completely characterized the graphs realizing
	those sets and developed an algorithm for constructing them but the sets $S_{\{1,j\}_n^2}$ which are conjectured to be not Laplacian realizable for
	large~$n$.
	\begin{conjecture}
		If $n\geqslant3$, then the only Laplacian realizable set of kind $S_{\{1,j\}_n^2}$ are $S_{\{1,3\}_4^2}$ and $S_{\{1,4\}_5^2}$.
	\end{conjecture}
	
	In addition, we believe that Lemma~\ref{compl.join.2} and Proposition~\ref{join.union} provide a way forward to find graphs realizing the
	sets~$S_{\{i,j\}_n^m}$ for other particular values of~$m$ from already existing values of $m$. For instance, the case~$m=1$ can be obtained from the case $m=n$ using Lemma~\ref{compl.join.2}. In a similar way, one can describe graphs realizing~$S_{\{i,j\}_n^m}$
	for $m=2$ from those with $m=n-1$ using Lemma~\ref{compl.join.2}. As well, Proposition~\ref{join.union} can help to find the repeated
	eigenvalue $m$ from previously known values of~$m$. However, it is not clear whether the use of Proposition~\ref{join.union} and
	Lemma~\ref{compl.join.2} may cover all the sets realizing $S_{\{i,j\}_n^m}$ for fixed $m$ and~$n$.
	
	\appendix
	
	\setcounter{equation}{0}
	\section{List of Laplacian integral graphs realizing $S_{\{i,n-1\}_n^1}$ up to order $8$}\label{Tables.const}

	In the paper by $K_n$, $P_n$, $K_{1,n-1}$ and $K_{p,q}~(p+q=n)$ we denote the complete graph, the path graph, the star graph and the complete bipartite graph on $n$~vertices, respectively. For the concepts and results about graphs not presented here, see, e.g.,~Bondy and Murty~\cite{Bondy-1976}, and Diestel~\cite{Diestel_2010}.

	In~\cite[p.~301--304]{CvetkovicRowlinson_2010} the authors found the Laplacian spectra of all graphs up to order~$5$. Also in~\cite{CvetkovicPetric_1984}, the authors depicted all the graphs of order $6$ without calculating their Laplacian spectra. We lists all the graphs realizing the sets~$S_{\{i,j\}_n^1}$ and $S_{\{i,j\}_n^2}$ up to order~$6$. Also, we list few graphs of order~$7$ and $8$ of such types. Note that we use the notation $A_{n}$ for the anti-regular graph of order~$n$, see, e.g.,~\cite{Merris.1_1994}.
	
	\vspace{4mm}
	
	\begin{center}
		\textbf{Table 1. Laplacian integral graphs realizing $S_{\{i,j\}_n^1}$ for $n=4,5,6,7,8.$}\\
	\end{center}
	\begin{center}
		\renewcommand*{\arraystretch}{1.9}
		\begin{longtable}{|c|c|c|c}
			\hline
			\textbf{Construction} & \textbf{Laplacian Spectrum} & $\boldsymbol{S_{\{i,j\}_n^m}}$\label{Table.1}\\
			\hline
			
			\endfirsthead
			\textbf{Construction} & \textbf{Spectrum} \\
			\hline
			\endhead
			\endfoot
			\hline
			\endlastfoot		
			$S_{4}\cong K_{3,1}$ &  ${\{0,1,1,4\}}$ & $S_{\{2,3\}_4^1}$\\
			\hline
			$K_{1}\vee\overline{K_{1,1,2}}$ & ${\{0,1,1,3,5\}}$&$S_{\{2,4\}_5^1}$\\
			\hline
			$K_{1}\vee(2K_{1}\cup P_3)$ &${\{0,1,1,2,4,6\}}$ & $S_{\{3,5\}_6^1}$\\
			\hline
			$K_{1}\vee(P_3\cup \overline{P_3})$&${\{0,1,1,2,3,4,7\}}$ & $S_{\{5,6\}_7^1}$\\
			\hline
			$K_{1}\vee[2K_{1}\cup(K_{1}\vee\overline{P_3})]$&${\{0,1,1,2,4,5,7\}}$ & $S_{\{3,6\}_7^1}$\\
			\hline
			$[(2K_{1}\vee \overline{P_3})\cup 2K_{1}]\vee K_{1}$&${\{0,1,1,3,4,5,6,8\}}$&$S_{\{2,7\}_8^1}$\\
			\hline
			$(A_5\cup 2K_{1})\vee K_{1}$&${\{0,1,1,2,3,5,6,8\}}$ &$S_{\{4,7\}_8^1}$\\
			\hline	
			$(\overline{P_3}\cup A_{4})\vee K_{1}$&${\{0,1,1,2,3,4,5,8\}}$ &$S_{\{6,7\}_8^1}$\\
			\hline
		\end{longtable}
	\end{center}
	
	\vspace{4mm}
	
	\begin{center}
		\textbf{Table 2. Laplacian integral graphs realizing $S_{\{i,j\}_n^2}$ for $n=4,5,6,7,8.$}\\
	\end{center}
	\begin{center}
		\renewcommand*{\arraystretch}{1.9}
		\begin{longtable}{|c|c|c|c}
			\hline
			\textbf{Construction} & \textbf{Laplacian Spectrum} & $\boldsymbol{S_{\{i,j\}_n^m}}$\label{Table.2}\\
			\hline
			
			\endfirsthead
			\textbf{Construction} & \textbf{Spectrum}  \\
			\hline
			\endhead
			\endfoot
			\hline
			\endlastfoot		
			$C_{4}$  & ${\{0,2,2,4\}}$&$S_{\{1,3\}_4^2}$\\
			\hline
			$K_{3,2}$ & ${\{0,2,2,3,5\}}$& $S_{\{1,4\}_5^2}$\\
			\hline
			${(K_{3,1}\cup K_{1})\vee K_{1}}$&${\{0,1,2,2,5,6\}}$&$S_{\{3,4\}_6^2}$\\
			\hline
			$(K_{3,1}\cup K_{2})\vee K_{1}$&${\{0,1,2,2,3,5,7\}}$&$S_{\{4,6\}_7^2}$\\
			\hline
			$[[(K_{2}\cup 2K_{1})\vee K_{1}]\cup K_{2}]\vee K_{1}$&${\{0,1,2,2,3,4,6,8\}}$ &$S_{\{5,7\}_8^2}$\\
			\hline
			$K_1\vee[K_1\cup (K_{1}\vee(2K_{1}\cup P_3))]$&${\{0,1,2,2,3,5,7,8\}}$ &$S_{\{4,6\}_8^2}$\\
			\hline
			
		\end{longtable}
	\end{center}
	
\section{Acknowledgements}
The work of M.\,Tyaglov was partially supported by National Natural Science Foundation of China under grant no.~11901384.

\end{document}